\documentclass[11pt]{article}
\PassOptionsToPackage{obeyspaces}{url}
\usepackage{amsfonts, amsmath, amssymb}
\usepackage[all]{xy}

\usepackage[hidelinks]{hyperref}
\usepackage{bookmark}
\bookmarksetup{open,numbered,depth=5} 

\makeatletter
\renewcommand\@seccntformat[1]{}
\makeatother

\usepackage{amsthm}

\usepackage{breakurl}
\usepackage{tikz}
\usetikzlibrary{calc}
\usetikzlibrary{decorations.pathmorphing}

\usepackage{etoolbox} 
\patchcmd{\thebibliography}{\leftmargin\labelwidth}{\leftmargin\labelwidth\addtolength\itemsep{-0.1\baselineskip}}{}{}

\usepackage{mathtools}

\usepackage{epstopdf}

\usepackage[justification=raggedright]{caption}

\author{Boris Bukh\thanks{Department of Mathematical Sciences, Carnegie Mellon University, Pittsburgh, PA 15213, USA\@. Supported in part by U.S.\ taxpayers through NSF grant DMS-2154063 and NSF CAREER grant DMS-1555149.}\and
R. Amzi Jeffs\thanks{Department of Mathematical Sciences, Carnegie Mellon University, Pittsburgh, PA 15213, USA\@. Supported by the National Science Foundation through Award No. 2103206.}}

\title{Planar convex codes are decidable}
\date{December 2022}

\usepackage[nameinlink]{cleveref}

\newtheorem{theorem}{Theorem}
\newtheorem{lemma}[theorem]{Lemma}
\newtheorem{corollary}[theorem]{Corollary}
\newtheorem{proposition}[theorem]{Proposition}

\newtheorem{definition}[theorem]{Definition}

\theoremstyle{remark}

\newtheorem{remark}[theorem]{Remark}

\newcommand*{\eqdef}{\stackrel{\mbox{\normalfont\tiny def}}{=}}  
\DeclarePairedDelimiter\abs{\lvert}{\rvert}                      
\newcommand*{\R}{\mathbb{R}}                                     
\newcommand*{\C}{\mathcal{C}}                                     
\newcommand*{\X}{\mathcal{X}}                                     
\newcommand*{\Y}{\mathcal{Y}}                                     
\newcommand*{\U}{\mathcal{U}}                                     
\newcommand*{\ipat}[2]{\operatorname{patt}_{#1}(#2)}               
\newcommand*{\simpli}[2]{\operatorname{sim}_{#2}(#1)}               
\DeclareMathOperator{\cl}{cl}                                 
\DeclareMathOperator{\interior}{int}                                 
\DeclareMathOperator{\conv}{conv}                                 
\DeclareMathOperator{\code}{code}                               
\DeclareMathOperator{\cone}{cone}                               

\tikzset{
  convex set/.style={very thick},
  filled convex set/.style={very thick,fill=white!75!black}
}
\begin{document}
\maketitle
\begin{abstract}
We show that every convex code realizable by compact sets in the plane admits a realization consisting of polygons, and analogously every open convex code in the plane can be realized by interiors of polygons. 
We give factorial-type bounds on the number of vertices needed to form such realizations. 
Consequently we show that there is an algorithm to decide whether a convex code admits a closed or open realization in the plane. 
\end{abstract}

\section{Introduction}\label{sec:introduction}
\paragraph{Prelude.}
A basic way to represent relationships between several sets is to use \emph{Venn diagrams}.
In a Venn diagram, the sets are represented by planar regions, and elements are represented
by points in such a way that the containment relation is preserved. For example,  \hyperref[fig:venn]{Figure 1(a)}
depicts the usual Venn diagram representing three sets $A,B$ and $C$, with eight regions
containing points of the eight sets ${A\cap B\cap C},\overline{A}\cap B\cap C,\dotsc,\overline{A}\cap \overline{B}\cap \overline{C}$
respectively. 

If certain intersections are empty, we may use Venn diagrams omitting them, as in \hyperref[fig:venn]{Figure 1(b)}.
\begin{figure}[h]
\begin{center}
\begin{tikzpicture}
\begin{scope}
\draw[convex set] (0,0) circle [radius=1cm];
\draw[convex set] (1,0) circle [radius=1cm];
\draw[convex set] (0.5,0.866025) circle [radius=1cm];
\node at (-1.2,-0.3) {$A$};
\node at (2.2,-0.3) {$B$};
\node at (-0.5,1.6) {$C$};
\node at (0,-1.5) {\textbf{a)} Eight regions are present};
\end{scope}
\begin{scope}[xshift=6cm]
\begin{scope}[scale=1.6,xshift=-0.4cm,yshift=-0.2cm]
\draw[convex set] (0.5,0) circle [x radius=1cm, y radius=0.2cm];
\draw[convex set] (0.25,0.433012) circle [x radius=1cm, y radius=0.2cm, rotate=60];
\draw[convex set] (0.75,0.433012) circle [x radius=1cm, y radius=0.2cm, rotate=-60];
\end{scope}
\node at (-0.9,0.4) {$A$};
\node at (1.26,0.4) {$B$};
\node at (0.1,-0.9) {$C$};
\node at (0,-1.5) {\textbf{b)} Region $A\cap B\cap C$ is omitted};
\end{scope};
\end{tikzpicture}
\end{center}
\vspace{-1.8em}
\caption{Three-set Venn diagrams.}\label{fig:venn}
\vspace{-1em}
\end{figure}

This paper is about Venn diagrams realizable by convex regions. These 
are also known as \emph{convex codes}, which is the terminology that we adopt.
Convex codes have been studied extensively over the last decade, with most work following a 2013 paper of Curto, Itskov, Veliz-Cuba and Youngs~\cite{CIVY}. 
The theory of convex codes can be viewed as a strengthening of the theory of $d$-representable complexes, which are the simplicial complexes that record nonempty intersections of a collection of convex sets in $\R^d$. 
Convex codes record not only the nonempty intersections, but also how other sets cover these intersections.

It is well known that there is an algorithm to determine whether or not a given simplicial complex is $d$-representable~\cite[Section 4.1]{tancer_survey}, but the analogous problem for convex codes has remained open. Chen, Frick and Shiu~\cite{CFS} showed that recognizing codes that can be realized by ``good covers" is undecidable, and asked about recognizing convex codes realizable by bounded open sets.  
Our main result is an affirmative answer to this question in the case $d=2$ (\Cref{thm:plane} below). 
Though our focus is on convex sets in the plane, we work in general $\R^d$ whenever
possible.

\paragraph{Convex codes.} Let $X_1,\dotsc,X_n$ be $n$ bounded convex sets in $\R^d$.
The \emph{intersection pattern} of the tuple $\X=(X_1,\dotsc,X_n)$ at the point $p\in \R^d$
is $\ipat{\X}{p}\eqdef \{i\in [n] \mid p\in X_i\}$. The \emph{convex code} of $\X$ is then defined
as $\code(\X)\eqdef \{\ipat{\X}{p} \mid p\in \R^d\}$. If $\C=\code(\X)$, we say that
$\C$ is a convex code \emph{realized} by~$\X$, and that $\X$ is a $d$-realization of~$\C$.
Franke and Muthiah showed that every $\C\subseteq 2^{[n]}$ containing $\emptyset$ is $d$-realizable for some $d$~\cite[Theorem~1]{all_codes_convex}.

We shall say that a realization $\X=(X_1,\dotsc,X_n)$ of a convex code $\C$ is \emph{closed} if every 
set $X_i$ is closed (which implies that it is compact). We similarly define an \emph{open} realization. If we consider only the realizations
in $\R^d$ for a fixed~$d\ge 2$, these are different concepts: Some convex codes admit only open
$d$\nobreakdash-realizations, some admit only closed $d$\nobreakdash-realizations, and some admit neither;
the first example showing the difference between open and closed $d$\nobreakdash-realizations is due to Lienkaemper, Shiu and Woodstock~\cite{LSW}, and
\cite{CGIK,jeffs_phenomena, jeffs_all_vectors} contain more recent results.

\paragraph{Our results.}
The primary motivation for this work is to find a way to tell which convex codes are $d$\nobreakdash-realizable.
To that end, it is natural to wonder whether every convex code admits a not-too-complicated realization.
We give a positive answer for $d=2$, namely that every closed convex code in the plane can be realized by polygons (some of which might degenerate into line segments or points), and that the number of vertices among these polygons is bounded by a computable function. 
We obtain the analogous result for open realizations in the plane as well.
As an immediate consequence, we conclude that there is an algorithm for determining if a convex code has a closed or open $2$\nobreakdash-realization.

\begin{theorem}\label{thm:plane}
Let $\C\subseteq 2^{[n]}$ be a convex code.
\begin{enumerate}
\item If $\C$ admits a closed convex realization in $\R^2$, then it admits a realization consisting of polygons with at most $6^n n!$ total vertices among them, and
\item If $\C$ admits an open convex realization in $\R^2$, then it admits a realization consisting of interiors of polygons with at most  $4\cdot 6^n (n+1)!$ total vertices among them. 
\end{enumerate}
As a consequence, there is an algorithm to decide whether or not $\C$ admits a closed (respectively, open) convex realization in the plane. 
\end{theorem}

Our proof consists of three steps. First, starting with an arbitrary realization, we shrink its sets as much as possible. Second, we show that the resulting inclusion-minimal realizations
are polygonal. Third, we bound the total number of vertices.
The first step works for any $d$, but our methods for the last two are specific to the plane. 
The situation for $d\ge 3$ remains unclear --- we do not know if convex codes in higher dimensions admit polytopal realizations, and we do not know if there exists
a computable bound on the number of vertices in inclusion-minimal polytopal realizations.
In particular, it is possible that there is no algorithm for deciding whether an arbitrary convex code is $d$\nobreakdash-realizable for fixed $d\ge 3$. 


\paragraph{Sharpness.} We do not know if the factorial-type bound in \Cref{thm:plane} is necessary. However, the number
of vertices must grow with the size of the code.
\begin{theorem}\label{thm:manyvertices}
Let $\X$ be a realization of $\C\subseteq 2^{[n]}$ consisting of closed polygons.
Then the total number of vertices among the polygons in $\X$ is at least $\abs{\C}/8n$.
\end{theorem}
In particular, since the code $2^{[n]}$ admits a realization by closed polygons \cite[Lemma~2.5]{local_obstructions}, it follows that
the $6^nn!$ in \Cref{thm:plane} cannot be improved to anything less than $\frac{1}{n}2^{n-3}$. 

\section{Step 1: Existence of inclusion-minimal realizations}\label{sec:inclusion-minimal}
Let $\X = (X_1, \ldots, X_n)$ be a tuple of compact convex sets. Recall that $\code(\X)\eqdef \{\ipat{\X}{p} \mid p\in \R^d\}$.
We say that $P\subseteq \R^d$ is a \emph{set of representatives} for $\X$ if $\code(\X)=\{\ipat{\X}{p} \mid p\in P\}$.
Obviously, every $\X$ admits a set of representatives of size $\abs{\code(\X)}\leq 2^n$ by picking a point in each region.

\begin{lemma}\label{lem:inclusion-minimal}
Let $\X = (X_1, \ldots, X_n)$ be a tuple of compact convex sets in~$\R^d$, and $P$ is a set of representatives for~$\X$.
Then there exists a tuple $\Y = (Y_1,\ldots, Y_n)$ of compact convex sets with the following properties:
\begin{enumerate}
\item \label{im:i} $Y_i\subseteq X_i$ for all $i\in[n]$,
\item \label{im:ii} $Y_i\cap P = X_i\cap P$ for all $i\in[n]$, 
\item \label{im:iii} $\code(\Y) = \code(\X)$, and
\item \label{im:iv} The tuple $\Y$ is inclusion-minimal among all tuples satisfying \ref{im:i}, \ref{im:ii} and \ref{im:iii}. 
\end{enumerate}
\end{lemma}

\begin{proof}
Let $\mathfrak{R}$ be the set consisting of tuples that satisfy conditions \ref{im:i}, \ref{im:ii} and \ref{im:iii}.
Introduce a partial order on $\mathfrak{R}$ by declaring $\Y\le \Y'$ if and only if $Y_i'\subseteq Y_i$ for all $i\in[n]$.
The poset $\mathfrak{R}$ is nonempty as it contains~$\X$. The condition \ref{im:iv} is equivalent to the tuple $\Y$ being a maximal element in this poset.
We shall find it using Zorn's lemma.

Given a chain $\mathfrak{C}$ of tuples in $\mathfrak{R}$, define a tuple $\Y = (Y_1, \ldots, Y_n)$ by 
$Y_i = \bigcap_{\Y'\in \mathfrak{C}} Y_i'$.
Each $Y_i$ is an intersection of compact convex sets, and so is itself compact and convex.
Moreover, each $Y_i$ satisfies \ref{im:i} and \ref{im:ii} since these properties are preserved by intersections. 
It remains to argue that $\Y$ satisfies condition \ref{im:iii}. 

Since $\ipat{\Y}{p}=\ipat{\X}{p}$ for every $p\in P$ and $P$ is a set of representatives for $\X$, it follows that
$\code(\X)\subseteq \code(\Y)$. Conversely, let $q$ be any point in $\R^d$. 
For every $i\notin \ipat{\Y}{q}$, there exists some $\Y^{(i)}\in \mathfrak{C}$ such that $q\notin Y^{(i)}_i$. For each $i\in [n]$,
fix such a $\Y^{(i)}$.
Let $\Y'$ be the smallest (with respect to inclusion) among the tuples~$\Y^{(1)},\dotsc,\Y^{(n)}$. Then $i\notin \ipat{\Y'}{q}$ whenever $i\notin \ipat{\Y}{q}$.
Since $Y_i\subseteq Y_i'$, we also have $i\in \ipat{\Y'}{q}$ whenever $i\in \ipat{\Y}{q}$. Therefore,
$\ipat{\Y}{q}=\ipat{\Y'}{q}$. Since $\Y'\in\mathfrak{C}$, this implies that $\ipat{\Y}{q}\in \code(\X)$.
Since $q\in \R^d$ is arbitrary, $\code(\Y) \subseteq \code(\X)$ and hence $\code(\Y)=\code(\X)$.
We conclude that every chain $\mathfrak{C}$ in $\mathfrak{R}$ has an upper bound, so $\mathfrak{R}$ has a maximal element, and the result follows. 
\end{proof}

\section{Step 2: Inclusion-minimal realizations in the plane are polygonal}\label{sec:polytopal}

\begin{lemma}\label{lem:local-arc}
Let $C\subseteq \R^2$ be a compact convex set with non-empty interior, and let $p$ be a boundary point of $C$.
Then there exists $\varepsilon >0$ so that every circle $S$ of positive radius $\delta < \varepsilon$ centered at $p$ 
intersects $C$ in a connected arc.
\end{lemma}
\begin{proof}
Consider the set 
\[
C_p \eqdef \{v\in \mathbb S^1\mid p+t v\in C \text{ for some } t >0\}.
\]
The set $C_p$ consists of the unit vectors which point ``into" $C$ from the boundary point $p$.
Since $C$ has non-empty interior, $C_p$ is a connected arc (not necessarily closed or open) that contains more than one point.
Up to affine transformation, we may assume that $p = (0,0)$, and that $C_p$ is an arc in the upper half of the unit circle.
Since $C_p$ might not be closed, it might not contain its endpoints. So, instead we select points
$q_1$ and $q_2$ in $C_p$ that are within angle $\pi/4$ of the left and right endpoints of~$C_p$, respectively.
For convenience, we may also choose $q_1$ and $q_2$ to have positive $y$-coordinates.
Finally, by scaling $C$ up, we may assume that both $q_1$ and $q_2$ lie in~$C$.

Let $S$ be a circle of radius $\varepsilon>0$ centered at~$p$.
The $\varepsilon$ is chosen small enough to satisfy two conditions, which we
call \textsc{disjointness} and \textsc{tangency}.
\begin{center}
\textbf{Disjointness: }The ball $\conv(S)$ is disjoint from the line segment~$q_1q_2$.
\end{center}
Consider the leftmost tangent from $q_1$ to $S$, and let $t_1=t_1(S)$ be the tangency point.
Define the point $t_2=t_2(S)$ symmetrically using the tangent from~$q_2$.
\begin{center}
\textbf{Tangency: }The angles $\angle q_1pt_1$ and $\angle q_2pt_2$ are both greater than $\pi/4$.
\end{center}
This situation is shown in \Cref{fig:local-arc}. 

\begin{figure}[h]
\[
\includegraphics{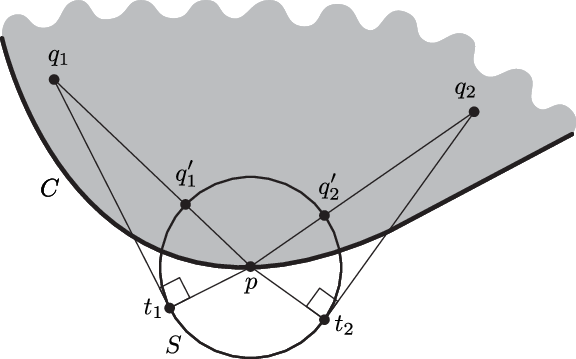}
\]
\caption{The small circle chosen in the proof of \Cref{lem:local-arc}, and the relevant tangent segments.}\label{fig:local-arc}
\end{figure}

Note that any smaller circle centered at $p$ also satisfies \textsc{disjointness} and \textsc{tangency}. We claim that any such circle intersects $C$ in a connected arc.
Without loss of generality, we prove this for~$S$.

Let $q_1'$ (respectively, $q_2'$) be the point where the line segment from $q_1$ (respectively, $q_2$) to $p$ crosses~$S$.
There are two arcs of $S$ connecting $q_1'$ and $q_2'$. One of the arcs is contained in the quadrilateral $\conv(q_1,q_2,q_1',q_2')$.
Since its vertices $q_1,q_2,q_1',q_2'$ lie in $C$, it follows that the arc is contained in $S\cap C$.
To see that $S\cap C$ is a connected arc, it will suffice to show that any $q'\in S\cap C$ is connected to $q_1'$ or $q_2'$ by an arc in $S\cap C$. 

Let $q'\in S\cap C$, and suppose that $q'$ does not lie between $q_1'$ and $q_2'$.
By the choice of $q_1$ and $q_2$, the point $q'$ is at the angle at most $\pi/4$ away from
either $q_1'$ or $q_2'$. Consider the first case (the other case being symmetric), i.e.,
$q'$ is to the left of $q_1'$. Since the angle $\angle q_1pq'$ is at most $\pi/4$, and the angle $\angle q_1pt_1$ exceeds
$\pi/4$, the point $q'$ lies in the cone bound by rays $\overrightarrow{pq_1}$ and $\overrightarrow{pt_1}$.
Since the line $q_1q'$ intersects $S$, it follows that $q'$ in fact lies inside the triangle $\bigtriangleup pq_1t_1$.
From this it follows that the triangle $\bigtriangleup pq_1q'$ is contained in $\bigtriangleup pq_1t_1$,
and so the angle $\angle pq'q_1$ is at least~$\angle pt_1q_1=\pi/2$. From this it follows
that the line segment $q_1q'$ intersects $S$ only in $q'$, and so its radial 
projection from $p$ onto $S$ is contained in $S\cap C$.
The projection is the desired arc from $q'$ to $q_1'$ in $S\cap C$.
\end{proof}

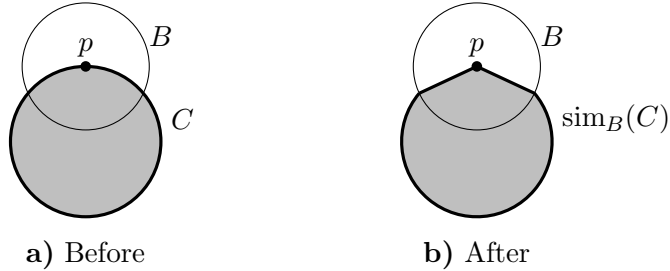
\begin{figure}[h]
\begin{center}
\begin{tikzpicture}
\def\bigangle{140}
\def\smallangle{400}
\def\distanceapart{2.6cm}

\def\drawsamestuff#1{%
\fill (0,1) circle (2pt);%
\node at (0,1.25) {$p$}; 
\draw let \p1=($(0,1)-(\bigangle:1cm)$),\n1={veclen(\x1,\y1)} in (0,1) circle [radius=\n1];
\node at (1,1.4) {$B$}; 
\node[anchor=west] at (1.0,0.3) {#1}; 
}
\begin{scope}[xshift=-\distanceapart];
\node at (0,-1.5) {\textbf{a)} Before}; 
\filldraw [filled convex set] (0,0) circle [radius=1cm]; 
\drawsamestuff{$C$}
\end{scope}
\begin{scope}[xshift=\distanceapart];
\node at (0,-1.5) {\textbf{b)} After}; 
\filldraw [filled convex set] (\bigangle:1cm) arc [radius=1cm, start angle=\bigangle, end angle=\smallangle] -- (0,1) -- (\bigangle:1cm); 
\drawsamestuff{$\simpli{C}{B}$}
\end{scope};
\end{tikzpicture}
\end{center}
\caption{Simplification of a convex set $C$ in the neighborhood of a point $p$.}
\end{figure}

For a set $A\subseteq \R^d$ and a point $p\in \R^d$, the \emph{cone over $A$ with apex $p$} is 
\[
  \cone_p(A)\eqdef \{\alpha p + (1-\alpha) x \mid 0\le \alpha \le 1 \text{ and } x\in A\}\cup\{p\}.
\]
(The reason for including $\{p\}$ in this definition is to ensure that $\cone_p(\emptyset)=\{p\}$.)

\begin{definition}\label{def:simplification}
Let $C\subseteq \R^2$ be a compact convex set, and let $p$ be a boundary point of $C$.
Let $B$ be a closed ball centered at $p$, and let $\partial B$ be the boundary of $B$. 
The \emph{simplification} of $C$ relative to $B$ is the set \[
\simpli{C}{B} \eqdef (C\setminus B) \cup \cone_p(\partial B\cap C). 
\]
\end{definition}

\begin{lemma}\label{lem:simplification}
  Let $B,C$ and $p$ be as in \Cref{def:simplification}. 
  If $\partial B\cap C$ is a connected arc, then $\simpli{C}{B}$ is the convex hull of $\{p\}\cup (C\setminus \interior B)$. 
\end{lemma}
\begin{proof}
The definition of $\simpli{C}{B}$ implies that $\simpli{C}{B} \subseteq \conv \bigl(\{p\}\cup (C\setminus \interior B)\bigr)$.
To show the reverse inclusion it is enough to argue that $\simpli{C}{B}$ is convex. To that
end, let $p_1,p_2\in \simpli{C}{B}$ be arbitrary, and consider two cases, 
according to how the line segment $p_1p_2$ is situated in relation to $\simpli{C}{B}$.
For brevity, denote the arc $\partial B\cap C$ by~$A$.\smallskip

\textbf{Suppose that both $p_1$ and $p_2$ lie in $\cone_p(A)$:} If one of them is $p$, the line segment $p_1p_2$ is contained in $\simpli{C}{B}$ by definition.
Otherwise, neither is equal to $p$ and we may consider the points $p_1'$ and $p_2'$ obtained by projecting $p_1$ and $p_2$ radially onto $\partial B$ from $p$. 
Since $p$ lies on the boundary of $C$, the angular measure of the arc $A$ is at most $\pi$.
As both $p_1'$ and $p_2'$ lie in $A$, this implies that the radial projection of the line segment $p_1p_2$ onto $\partial B$ is an arc in $A$ connecting $p_1'$ to $p_2'$.
So, the line segment $p_1p_2$ lies in $\simpli{C}{B}$.\smallskip

\textbf{Suppose one of $p_1$ and $p_2$ does not lie in $\cone_p(A)$:} 
The line segment $p_1p_2$ lies in $C$, and so we need prove only that any portion of it that lies in $B$ also lies in $\cone_p(A)$.
Note that if $p_1p_2\cap B$ is nonempty, then it is a line segment whose endpoints lie in $\cone_p(A)$. 
In our analysis of the first case, we showed that such a line segment is contained in $\cone_p(A)$. 
Hence $p_1p_2$ is contained in $\simpli{C}{B}$ and the result follows.
\end{proof}

\begin{proposition}\label{prop:polygonal}
Let $\X = (X_1,\ldots, X_n)$ be a tuple of compact convex sets in $\R^2$, and let $P$ be a finite set of representatives for $\X$.
If $\Y = (Y_1, \ldots, Y_n)$ is an inclusion-minimal realization of $\code(\X)$ as in \Cref{lem:inclusion-minimal}, then each $Y_i$ is a polygon. 

\end{proposition}
\begin{proof}
Recall that $x$ is an \emph{extreme point} of a convex set $C$ if $x$ is not contained in $\conv(C\setminus \{x\})$. 
Every compact convex set is the convex hull of its extreme points, and so it will suffice to show that every $Y_i$ has finitely many extreme points.
Suppose for contradiction that some $Y_i$ has infinitely many extreme points. 
Each extreme point must lie on the boundary of $Y_i$, and since the boundary is compact there must exist a point $p$ on the boundary of $Y_i$ which is a limit point of the extreme points of $Y_i$.
Fix this $i$ and the point $p$ for the remainder of the proof.

Choose a closed ball $B$ centered at $p$ whose radius is small enough as to satisfy the following four conditions for each $j\in [n]$:
\begin{itemize}
\item $B$ is disjoint from $\partial Y_j$ unless $p\in \partial Y_j$,
\item if $Y_j$ has a non-empty interior and $p\in \partial Y_j$ then $B$ satisfies the conclusion of \Cref{lem:local-arc}, i.e., $\partial B\cap Y_j$ is a connected arc spanning an angle of no more than $\pi$,
\item if $Y_j$ is a line segment containing $p$, the set $\partial B\cap Y_j$ contains at least one point,
\item $B$ is disjoint from $P\setminus \{p\}$.
\end{itemize}
Now let $\Y' = (Y_1', \ldots, Y_n')$ where \[
Y_j' \eqdef \begin{cases} Y_j &\text{if $p\notin \partial Y_j$}, \\
\simpli{Y_j}{B} & \text{if $p\in \partial Y_j$}.
\end{cases}
\]
Note that if $Y_j$ is a line segment, then $\partial Y_j=Y_j$ and so $Y_j'=\simpli{Y_j}{B}$.
Clearly $Y_j'\subseteq Y_j$ for every~$j\in [n]$.
Moreover, the simplification $Y_i'=\simpli{Y_i}{B}$ is a proper subset of $Y_i$ as $p$ is no longer a limit point of the extreme
points.

We claim that $\Y'$ is a closed convex realization of $\code(\Y)$. 
Observe that each $Y_j'$ is compact and convex, for the sets $\simpli{Y_j}{B}$ are convex hulls of compact sets by \Cref{lem:simplification}.
It remains to argue that $\code(\Y') = \code(\Y)$. 
For this, consider any point $q\in \R^d$.
If $q\notin B$ or $q = p$, then clearly $\ipat{\Y'}{q} = \ipat{\Y}{q}$.
Otherwise, consider the radial projection centered at $p$ onto $\partial B$.
Let $q'\in \partial B$ be the projection of the point~$q$.
From the definition of simplification, $q\in Y_j'$ if and only if $q'\in Y_j$. 
Thus we have $\ipat{\Y'}{q} = \ipat{\Y}{q'}$, and $\code(\Y')\subseteq \code(\Y)$.
For the reverse containment, recall that $B$ avoids our set of representatives for $\Y$, except possibly $p$.
The intersection pattern at $p$ is the same in $\Y'$ as in $\Y$, so every member of $\code(\Y)$ appears in $\code(\Y')$.
Thus $\code(\Y') = \code(\Y)$.
However, $\Y'$ satisfies (i), (ii) and (iii) of \Cref{lem:inclusion-minimal}, and so the fact that $Y_i'$ is a proper subset of $Y_i$ contradicts our choice of $\Y$ as an inclusion-minimal realization. 
Thus $\Y$ must be polygonal, and the result follows.
\end{proof}

\section{Step 3: Bounding the number of vertices}
Our main task in this section is to prove the first part of \Cref{thm:plane}, which says that every convex code admitting a closed $2$\nobreakdash-realization admits a polygonal
realization in which the number of vertices is bounded by an explicit function of the code size. Indeed, one can express the question of whether
a convex code $\C$ is $2$-realizable by polygons with at most $m$ vertices as a sentence in the language of fields. As the theory of real closed
fields is decidable (originally shown by Tarski \cite{tarski}, see for example \cite{basu_pollack_roy} for a modern exposition), this will imply
\Cref{thm:plane}. 

Let $\X=(X_1,X_2,\dotsc,X_n)$ be a closed $2$\nobreakdash-realization of a convex code~$\C$.
Let $P$ be a set of representatives for $\X$ of size $\abs{P}\leq 2^n$.
We may assume that $\X$ is inclusion-minimal with respect to $P$, as in \Cref{lem:inclusion-minimal}, which by \Cref{prop:polygonal} implies that each $X_i$ is a closed polygon.
We shall show that, if $\X$ has too many vertices, then we may shrink $\X$ by pulling a vertex inward.
As $\X$ is inclusion-minimal, a bound on the number of vertices will follow.

\paragraph{Locating good vertices.} The boundary of each $X_i$ is made of \emph{vertices} and \emph{edges} that join them. We adopt the convention
that the edges are relatively open, i.e., \emph{edges do not contain their endpoints}.
A point $p$ is \emph{good} with respect to $(X_1,\dotsc,X_n)$ if $p$ does not lie on an edge of any~$X_i$.
A point that is not good is called \emph{bad}. 
A point $p$ is \emph{very good} with respect to $(X_1,\dotsc,X_n)$ if it is good and $p\notin X_1\setminus \operatorname{vertices}(X_1)$.
\begin{lemma}\label{lem:goodvertices}
  Suppose $X_1,\dotsc,X_n$ are convex polygons in $\R^2$ and $X_1$ has $K$ vertices. Then there is $j$ such that $X_j$ contains
  at least $\frac{1}{3^{n-1}(n-1)!} K$ very good vertices with respect to the tuple $(X_1,\dotsc,X_n)$.
\end{lemma}
\begin{proof}
  We argue by induction on~$n$. The base case $n=1$ is trivial, so assume that $n\geq 2$. If at least $\frac{1}{3^{n-1}(n-1)!} K$
  vertices of $X_1$ are good, then we are done with $j=1$. Otherwise, let $B$ be the set of all bad vertices in $X_1$. Note that
  $\abs{B}\geq \bigl(1-\frac{1}{3^{n-1}(n-1)!}\bigr)K\geq \tfrac{2}{3}K$. For each $i\neq 1$, let $B_i$ be the set of those points in $B$
  that lie on some edge of~$X_i$. Since $B=\bigcup_{i\geq 2} B_i$, it follows from the pigeonhole principle that one of $B_2,\dotsc,B_n$
  is large. Say, without loss of generality, that $\abs{B_2}\geq \tfrac{2}{3}K/(n-1)$.

\begin{figure}[h]
\[
\includegraphics{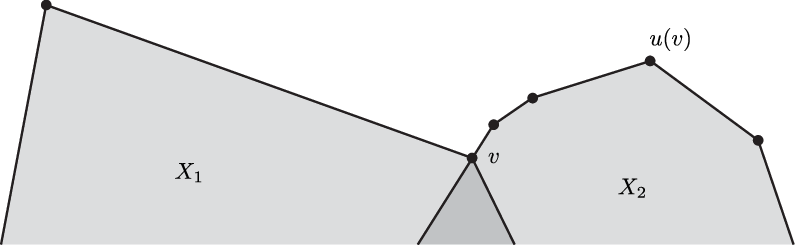}
\]
\caption{Associating bad vertices of $X_1$ to vertices of $X_2'$.}\label{fig:uv}
\end{figure}

  For each $v\in B_2$ choose a direction for which $v$ is an extremal point of~$X_1$.
  Let $u(v)$ be the vertex of $X_2'\eqdef \conv(X_1\cup X_2)$ that is extremal for this direction 
  (if there are two such vertices, pick arbitrarily). Note that $u(v)$ is a vertex of $X_2$
  that lies between two consecutive intersection points of $\partial X_1$ with~$\partial X_2$.
  So, the map $v\mapsto u(v)$ from $B_2$ to $\operatorname{vertices}(X_2')$ is at most two-to-one. See \Cref{fig:uv}.

  It follows that the polygon $X_2'$ has at least $\abs{B_2}/2\geq \frac{K}{3(n-1)}$ vertices.
  By the induction hypothesis applied to the $(n-1)$-tuple $(X_2',X_3,X_4,\dotsc,X_n)$ there is a $j\in\{2,3,\dotsc,n\}$ and
  a set of $\frac{1}{3^{n-2}(n-2)!}\cdot \frac{K}{3(n-1)}$ very good vertices in $X_j$ (if $j>2$) or $X_2'$ (if $j=2$) with
  respect to this $(n-1)$-tuple. Since $\frac{1}{3^{n-2}(n-2)!}\cdot \frac{K}{3(n-1)}=\frac{1}{3^{n-1}(n-1)!}K$, the proof is complete.
\end{proof}
\begin{corollary}\label{cor:goodpair}
Suppose $\X=(X_1,\dotsc,X_n)$ is an $n$-tuple of convex polygons in the plane, 
and $P$ is a finite set of representatives for $\X$. If one of $X_1,\dotsc,X_n$ has more than $3^n(n-1)!\abs{P}$ vertices,
then there exist two distinct $v,v'\in \R^2$ such that
\begin{itemize}
\item[(i)] $v$ is a vertex of some polygon in $\X$, and
\item[(ii)] Both $v$ and $v'$ are good with respect to $\X$, and
\item[(iii)] $v,v'\notin P$, and
\item[(iv)] $\ipat{\X}{v}=\ipat{\X}{v'}$.
\end{itemize}
\end{corollary}
\begin{proof}
By the previous lemma, some $X_j$ contains more than $3\abs{P}$ good vertices with respect to~$\X$.
More than $2\abs{P}$ among them satisfy (iii). By the pigeonhole principle, a pair satisfies (iv) as well.
\end{proof}

\paragraph{Pulling a vertex.} Let $\X$, $v$ and $v'$ be as in \Cref{cor:goodpair}. We will argue that $\X$ is not inclusion-minimal. By permuting the polygons in $\X$,
we may assume that $v$ is a vertex of $X_1,X_2,\dotsc,X_m$ and is not a vertex of $X_{m+1},\dotsc,X_n$,
for some $m\geq 1$. Let $B$ be a ball around~$v$ small enough so that the only polygon vertex inside $B$ is~$v$. We shall shrink each of $X_1,X_2,\dotsc,X_m$ by modifying
inside~$B$. Specifically, move the vertex $v$ slightly toward $v'$ along the line segment~$vv'$ to obtain a new vertex~$\tilde{v}$.
Then, for each $i\in [m]$, let $a_i$ and $b_i$ be the intersection points between the boundaries of $X_i$ and $B$.
If $X_i$ is a line segment, then the points $a_i$ and $b_i$ are equal.
Let $\widetilde{X}_i$ be obtained from $X_i$ by replacing the line segments $a_iv$ and $b_iv$ with
$a_i\tilde{v}$ and $b_i\tilde{v}$ respectively.

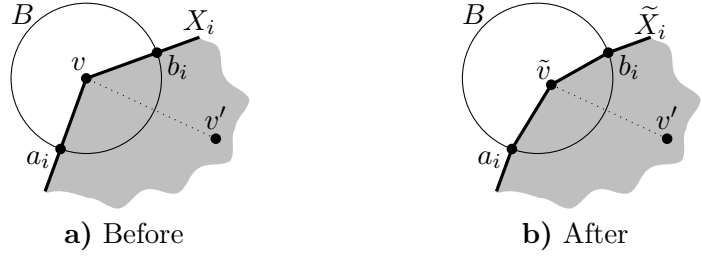
\begin{figure}[h]
  \begin{center}
  \begin{tikzpicture}
    \def\seglen{0.79cm}  
    \def\angleA{250}
    \def\angleB{20}
    \def\slen{1.6cm}
    \def\angleVprime{-25}
    \def\vprime{\angleVprime:1.9cm}
    \def\vtilde{\angleVprime:0.2cm}
    \def\distanceapart{3cm}
    \def\labeloffset{0.5cm}
    \def\drawcommon#1{%
    \draw (0,0) circle [radius=1cm];
    \fill (\angleA:1cm) circle (2pt) node[below left, yshift=3pt] {$a_i$};
    \fill (\angleB:1cm) circle (2pt) node[below right, yshift=3pt] {$b_i$};
    \node at (\angleB:\slen) [yshift=0.2cm] {$#1$};
    \fill (\vprime) circle (2pt) node [above] {$v'$};
    \node at (135:1.2cm) {$B$};
    }
    \begin{scope}[xshift=-\distanceapart];  
    \fill [filled convex set] (\angleA:\slen) -- (0,0) -- (\angleB:\slen) decorate[decoration={snake,segment length=\seglen}]  { .. controls (3,-1) and (0.4,-1.7)  .. cycle};
    \drawcommon{X_i}
    \draw[convex set] (\angleA:\slen) -- (0,0) -- (\angleB:\slen);
    \fill (0,0) circle (2pt) node [shift={(-0.1,0.2)}] {$v$};
    \draw[dotted] (\vprime) -- (0,0);
    \node at (\labeloffset,-2.1) {\textbf{a)} Before};
    \end{scope}
    \begin{scope}[xshift=\distanceapart];  
    \fill [filled convex set] (\angleA:\slen) -- (\angleA:1cm) -- (\vtilde) -- (\angleB:1cm) -- (\angleB:\slen) decorate[decoration={snake,segment length=\seglen}]  { .. controls (3,-1) and (0.4,-1.7)  .. cycle};
    \drawcommon{\widetilde{X}_i}
    \draw[convex set] (\angleA:\slen) -- (\angleA:1cm) -- (\vtilde) -- (\angleB:1cm) -- (\angleB:\slen);
    \fill (\vtilde) circle (2pt) node [shift={(-0.1,0.2)}] {$\tilde{v}$};
    \draw[dotted] (\vprime) -- (\vtilde);
    \node at (\labeloffset,-2.1) {\textbf{b)} After};
    \end{scope}
  \end{tikzpicture}
\caption{Pulling the vertex $v$ toward $v'$ in $X_i$.}
\end{center}
\end{figure}

Since $\ipat{\X}{v}=\ipat{\X}{v'}$ and $v\in X_1,\dotsc,X_m$, it follows that $v'\in X_1,\dotsc,X_m$ as well.
Thus, the sets $\widetilde{X}_1,\dotsc,\widetilde{X}_m$ are convex as long as $\tilde{v}$ is a sufficiently
small perturbation of~$v$.
Now consider the tuple $\widetilde{X}=(\widetilde{X}_1,\dotsc,\widetilde{X}_m,X_{m+1},\dotsc,X_n)$.
Note that we may obtain $B\cap \widetilde{X}_1,\dotsc,B\cap \widetilde{X}_m$ from $B\cap X_1,\dotsc,B\cap X_m$ 
by a continuously deforming $B$ while keeping its boundary fixed. Hence,
the $m$-tuples $B\cap X_1,\dotsc,B\cap X_m$ and $B\cap \widetilde{X}_1,\dotsc,B\cap \widetilde{X}_m$
realize the same convex code. Furthermore, since the vertex $v$ is good, the ball $B$ can be
chosen so small that that it does not intersect the boundary of $X_{m+1},X_{m+2},\dotsc,X_n$,
and so $\code(\X)=\code(\X')$. As $P\cap B=\emptyset$ for sufficiently small $B$,
we conclude that $\X$ is not inclusion-minimal. 

We immediately obtain the following corollary.

\begin{corollary}\label{cor:part(i)}
Let $\Y$ be an inclusion-minimal $2$\nobreakdash-realization of a convex code on $[n]$ with respect to a set of representatives $P$, as in \Cref{lem:inclusion-minimal}. 
Then each set in $\Y$ is a polygon with no more than $3^n(n-1)!|P|$ vertices. 
\end{corollary}

As we may choose $|P|\le 2^{n}$, we conclude that every convex code that admits a closed convex realization in the plane admits a realization by polygons with at most $6^n(n-1)!$ vertices each.

\paragraph{Closed \texorpdfstring{$2$}{2}-realizability is decidable.} Let $N=6^n(n-1)!$.
Every polygon with at most $N$ vertices is an intersection of $N$ closed halfspaces.
Writing a polygon
$X_i$ as $\bigcap_{j=1}^{N} \{a_{i,j}x+b_{i,j}y\leq c_{i,j}\}$ with $a_{i,j},b_{i,j},c_{i,j}\in \R$, we can write
a formula asserting that a point $(x,y)$ is in $X_i$,
\[
\Psi_i(\vec{a}, \vec{b}, \vec{c}, x, y) \eqdef \forall j\in[N]\, (a_{i,j} x + b_{i,j} y \le c_{i,j}), 
\]
where we abbreviated $\vec{a}\eqdef (a_{i,j})_{i\in [n],j\in [N]}$, and similarly for $\vec{b}$ and $\vec{c}$.
We can then write a formula asserting that $S\in \code(X_1,\dotsc,X_n)$,
\[
\Phi(S;\vec{a},\vec{b},\vec{c}\,)\eqdef \exists x,y\, \forall i \in [n] \Big(i \in S \, \Leftrightarrow \, \Psi_i(\vec{a}, \vec{b}, \vec{c}, x,y)\Big).
\]
We can also write a formula asserting that $X_i$ is bounded,
\[
  \Gamma_i(\vec{a},\vec{b},\vec{c}\,)\eqdef \exists r\,\Bigl(x^2+y^2>r \implies \neg \Psi_i(\vec{a}, \vec{b}, \vec{c}, x, y)\Bigr).
\]
Then the assertion that a convex code $\C\subseteq 2^{[n]}$ admits a closed $2$\nobreakdash-realization can be written as
\[
  \exists \vec{a},\vec{b},\vec{c} \bigwedge_{i\in [n]} \Gamma_i(\vec{a},\vec{b},\vec{c}\,)\wedge \bigwedge_{S\in \C} \Phi(S;\vec{a},\vec{b},\vec{c}\,)\wedge \bigwedge_{S\notin \C} \neg\Phi(S;\vec{a},\vec{b},\vec{c}\,).
\]
As mentioned above, the theory of real closed fields is decidable, and hence so is the question of whether a convex code
admits a closed $2$\nobreakdash-realization.

\begin{remark}
Kunin, Lienkaemper and Rosen~\cite{KLR} were the first to explain that, for any fixed computable function of $n$ and $d$, there is an algorithm to decide whether a convex code has a polytopal $d$\nobreakdash-realization with the total number of vertices bounded by this function (in fact they considered a bound on the number of facets, but this is the same up to a slight modification of the computable function in question).
Seeking a combinatorial argument that $d$-realizability is decidable, they reduced the problem of finding a $d$\nobreakdash-realization with a bounded number of vertices to the problem of determining whether certain oriented matroids are representable.
We have opted for the direct statements above, since deciding representability of oriented matroids is complete with respect to the theory of real closed fields, and hence the reduction to oriented matroids does not improve the computational complexity of the decision problem.
\end{remark}

\section{Complexity: Realizations requiring many vertices}\label{sec:many}
In this section we prove \Cref{thm:manyvertices}, which gives a lower bound on the number
of vertices in a realization $\X$ of any code $\C$ by closed polygons.

Let $\ell$ be some vertical line. The intersection pattern of $\X$ at $p\in \ell$ 
is determined by the position of $p$ relative to the intersection points of $\partial X_1,\dotsc,\partial X_n$ with $\ell$. 
Let $P(\ell)$ be the set of intersection patterns that occur on $\ell$. 

Imagine starting with $\ell$ to the left of all $X_1,\dotsc,X_n$, and sweeping $\ell$ to the right across the plane. The relative positions of the intersection points between $\ell$ and the various $\partial X_i$ changes only as $\ell$ passes through some certain \emph{special} points: leftmost and rightmost points of the polygons, and the points
where two polygon edges intersect transversely. By rotating $\R^2$ slightly, we may ensure that the $x$-coordinates
of special points are distinct. However, it might still happen that a point $p$ is special for several different reasons: for example,
if several different pairs of polygons intersect at~$p$. In this case, we speak of \emph{multiplicity} of $p$,
and denote it by $m(p)$. 

As the line $\ell$ passes through a point $p$ of multiplicity $m(p)$, the set $P(\ell)$ changes, and we may obtain intersection patterns
that did not appear to the left of $p$. We obtain a new
pattern when $\ell$ reaches $p$, and then at most one new pattern per each polygonal segment through $p$ once $\ell$ passes over~$p$.
Hence, the total number of new patterns is at most $1+2m(p)\leq 4m(p)$.
Letting $M$ be the number of special points (with multiplicity), it follows that $\abs{\C}\leq 4M$.

Let $N$ be the total number of vertices among the polygons $X_1,\dotsc,X_n$. Since a line segment
belonging to one of $\partial X_1,\dotsc,\partial X_n$ intersects $\partial X_1,\dotsc,\partial X_n$
transversely in a total of at most $2n-2$ points, it follows
that $M\leq 2n+(2n-2)N$. Therefore, $\abs{\C}\leq 8nN$, as desired.

\section{Bonus: From closed to open realizations}\label{sec:open}

In this section we will prove the second part of \Cref{thm:plane}. 
We first show that when we replace sets by their closures in an open realization, new intersection patterns can only arise inside convex regions with empty interior. 

\begin{lemma}\label{lem:empty-interior}
Let $\U = (U_1, \ldots, U_n)$ be a tuple of convex sets in $\R^d$, and define $\X = (X_1, \ldots, X_n)$ where $X_i \eqdef \cl U_i$. 
If $c\in \code(\X)\setminus \code(\U)$, then $\bigcap_{i\in c} X_i$ has empty interior. 
\end{lemma}
\begin{proof}
Suppose for contradiction that there exists $c\in \code(\X)\setminus \code(\U)$ so that $U\eqdef \interior \left(\bigcap_{i\in c} X_i\right)$ is nonempty.
Let $p\in U$, and let $q$ be a point with $\ipat{\X}{q} = c$.
Since $q\in \bigcap_{i\in c} X_i$, the line segment $pq$ is contained in $U$ except possibly for the point~$q$. 
Since all $X_i$ are closed, $q$ lies a positive distance from $X_i$ for all $i\notin c$. 
Hence we may choose $q' \in U$ on the segment $pq$ so that $\ipat{\X}{q'} = c$.
Since $U_i$ is convex, $\interior X_i\subseteq U_i$.
Hence, $q'\in \interior X_i \subseteq U_i$ for all $i\in c$, and so $\ipat{\U}{q'} = c$.
This contradicts the assumption that $c\notin \code(\U)$.
\end{proof}

\paragraph{Reduction from open to closed realizations.} 
Let $\U = (U_1, \ldots, U_n)$ be a tuple of bounded convex open sets in $\R^2$.
Let $\X = (X_1, \ldots, X_n)$ be the tuple of compact convex sets in $\R^2$ with $X_i \eqdef \cl U_i$ for all $i\in[n]$.
For every nonempty $\sigma\subseteq [n]$ for which $\bigcap_{i\in\sigma} X_i$ is nonempty but has empty interior, let $L_\sigma$ be a line containing the set $\bigcap_{i\in\sigma} X_i$.
Fix a set of representatives $P$ for $\X$ with the following properties:
\begin{enumerate}
\item \label{red:p} for every $c\in \code(\U)$, the set $P$ contains the vertices of a triangle contained in $\bigcap_{i\in c} U_i$ whose interior contains a point $p$ with $\ipat{\U}{p} = c$, 
\item \label{red:u} for every pair $(L_\sigma, X_i)$ with $L_\sigma \cap U_i\neq \emptyset$, the set $P$ contains the vertices of a quadrilateral contained in $X_i$ which has one of its diagonals equal to $L_\sigma \cap X_i$. 
\end{enumerate}
To satisfy the conditions \ref{red:p} and \ref{red:u}, first select, for each $c\in \code(\U)$, a point $p$ such that $\ipat{\U}{p}=c$, and then choose
an appropriate small triangle or quadrilateral containing~$p$. This way we may guarantee 
\begin{align*}
  |P| &\le \abs{\code(\X)}+3\abs{\code(\U)}+4n2^n \\&\le 2^n+3\cdot 2^n + 4n2^n = 4(n+1)2^{n}.
\end{align*}

\parshape=13 0cm.82\hsize 0cm.82\hsize 0cm.82\hsize 0cm.82\hsize 0cm.82\hsize 0cm.82\hsize 0cm.82\hsize 0cm.82\hsize 0cm.82\hsize 0cm.82\hsize 0cm.82\hsize 0cm\hsize 0cm \hsize
Now let $\X' = (X_1', \ldots, X_n')$ be an inclusion-minimal realization of $\code(\X)$ with respect to the set of representatives $P$, as in \Cref{lem:inclusion-minimal}.
By \Cref{cor:part(i)} the number of vertices among all $X_i'$ is at most $3^n (n-1)! |P| \le 4(n+1) \cdot 6^n (n-1)!$. 
Define $\U' = (U_1', \ldots, U_n')$ to be the tuple with $U_i'\eqdef \interior X_i'$ for all $i\in[n]$.
We claim that $\code(\U') = \code(\U)$, which will establish the second part of \Cref{thm:plane}. 
The relationships between the tuples $\U,\X,\X'$ and $\U'$ is depicted on the right.
\vadjust{\hfill\smash{\raise 70pt\llap{\xymatrix@C=3.3em{%
\X\ar[r]^{\operatorname{incl. min.}}&\X'\ar[d]^{\interior}\\
\U\ar[u]^{\cl}&\U'%
}}}}

\subparagraph{The inclusion \texorpdfstring{$\code(\U')\subseteq \code(\U)$}{"U' is subset of U''}:}
Note that the condition \ref{red:u} implies that 
\begin{equation}\label{same_intersection}
  L_\sigma \cap U_i' = L_\sigma\cap U_i\text{ for every }i\in[n].
\end{equation}
Indeed, let $a,b$ be the end points of the line segment $L_\sigma\cap X_i$.
Let $c,d$ be the other two vertices of the quadrilateral from the condition \ref{red:u}.
Suppose first that $p$ is an interior point of the line segment $L_\sigma\cap X_i$.
Then $p\in \interior \conv \{a,b,c,d\}$. Since $a,b,c,d\in X_i\cap P$, it follows
that $a,b,c,d\in X_i'\cap P$ and hence $p\in \interior \conv \{a,b,c,d\}\subseteq U_i'$.
Suppose next that $p$ is either $a$ or $b$.
In this case, since $U_i$ is open it follows that $p\notin U_i$.
Also, $p\in \partial X_i'$ and so $p\notin U_i'$.
So, neither $a$ nor $b$ is in $U_i'$ implying that the intersection of the line
$L_\sigma$ with $U_i'$ is the open line segment $(a,b)$. Because
$L_\sigma\cap U_i=(a,b)$, this proves~\eqref{same_intersection}.\smallskip

From \eqref{same_intersection} it follows that $\ipat{\U'}{p} = \ipat{\U}{p}$ if~$p\in L_{\sigma}$. 
Let $L=\bigcup_{\sigma} L_{\sigma}$ be the union of all $L_{\sigma}$'s. It remains to show that $\ipat{\U'}{p} \in \code(\U)$ for $p\notin L$.
So, fix a point $p\notin L$, and put $c\eqdef \ipat{\X'}{p}$.
Since any intersection pattern in $\code(\X)\setminus \code(\U)$ arise only inside $L$ by \Cref{lem:empty-interior},
it follows any intersection patterns in $\code(\X')\setminus \code(\U)$ also arise only inside $L$.
As $p\notin L$, this implies that $c\in \code(\U)$. 
So by the choice of $P$ in the condition \ref{red:p} the set $\bigcap_{i\in c} X_i'$ has nonempty interior. 
Let $q$ be a point in the interior of $\bigcap_{i\in c} X_i'$.
For sufficiently small $\varepsilon > 0$, consider the point $p_\varepsilon \eqdef p+ \varepsilon(p-q)$.
Since $q\in \interior X_i'$ for every $i\in c$, it follows that $p_{\varepsilon}\in X_i'$ if
and only if $p_{\varepsilon}\in \interior X_i'$, for small enough~$\varepsilon$.
As $\interior X_i'=U_i'$, we conclude that $\ipat{\U'}{p}=\ipat{\X'}{p_\varepsilon}$.
Furthermore, since $p\notin L$ and $\varepsilon>0$ is small, we also
have $p_\varepsilon\notin L$. Again using the fact that the intersection patterns in $\code(\X')\setminus \code(\U)$ arise only inside~$L$, we have  $\ipat{\U'}{p} = \ipat{\X'}{p_\varepsilon} \in \code(\U)$,
as promised.

\subparagraph{The inclusion \texorpdfstring{$\code(\U)\subseteq \code(\U')$}{"U is a subset of U' "}:}
Let $c\in \code(\U)$. Consider the point $p$ and the triangle $\bigtriangleup$ in $\bigcap_{i\in c} U_i$ from the condition~\ref{red:p}.
Since $P$ contains the vertices of $\bigtriangleup$, the same is true of the $X_i'$ for $i\in c$.
Because $p\in \interior \bigtriangleup$ it follows that $p\in U_i'$ for all $i\in c$.
As $U_i'\subseteq U_i$ for all $i\in [n]$, we also have $p\notin U_i'$ for all $i\notin c$.
Therefore, $\ipat{\U'}{p} = \ipat{\U}{p}$, and so $c\in \code(\U')$.

\paragraph{Open \texorpdfstring{$2$-realizability}{2-realizability} is decidable.}
This is similar to the decidability of closed $2$-realizability. There are two minor differences.
First, the number of polygon vertices is bounded by  $4(n+\nobreak 1)\cdot 6^n(n-1)!$ instead of $6^n(n-1)!$.
Second, the formula $\Psi_i$ should use strict inequalities:
\[
\Psi_i(\vec{a}, \vec{b}, \vec{c}, x, y) \eqdef \forall j\in[N]\, (a_{i,j} x + b_{i,j} y < c_{i,j}). 
\]
The rest of the argument is the same. 

\section{Epilogue: comments about the general case}
In this paper we restricted our attention to realizations by open or closed bounded sets in the plane. Extending our
techniques to arbitrary convex sets in the plane does not seem to pose serious challenges, but would make the argument more technical.
However, the situation in higher dimensions appears more difficult, perhaps fundamentally so.

For example, there exist convex sets in $\R^3$ that are not polyhedral at a neighborhood of a point, but
which cannot be made locally polyhedral using the simpification process in \Cref{def:simplification} centered at the point.
Namely, consider the set $C$, which is the convex hull of
\[
\{(x,y,z) \mid z=0, x^2+ y^2= 1\}\cup \{(x,y,z) \mid y=0, x^2+z^2 = 1\}.
\]
The set $C$ is the convex hull of two circles which pass through the point $p=(1, 0, 0)$ transversely (see \Cref{fig:bad-boundary-point}). 
If $B$ is a small ball centered at $p$, then $\cone_p(\partial B \cap C)$ is not even a convex set, no matter how small $B$ is. 
One could attempt to remedy this by replacing $B$ with a neighborhood of $p$ that is shaped differently than a ball, but the fact that we need a simplification process that works simultaneously for many sets introduces further difficulties. 
Even if one does obtain polytopal realizations in higher dimensions, a computable bound on the number of vertices is not immediate, as our \Cref{lem:goodvertices} does not have an obvious generalization beyond the plane. 

\begin{figure}
\[
\includegraphics{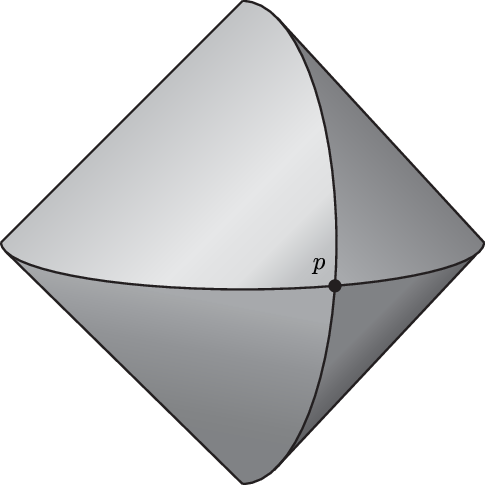}
\]
\caption{A convex set in $\R^3$ that is difficult to locally simplify.}
\label{fig:bad-boundary-point}
\end{figure}

These difficulties suggest that recognizing $3$-realizable convex codes is harder than recognizing $2$\nobreakdash-realizable convex codes.
The results of Kunin, Lienkaemper and Rosen~\cite{KLR} suggest that recognizing $2$\nobreakdash-realizable convex codes is already at least NP-hard. In addition,
Tancer~\cite[Section 6]{NPhard} showed that recognizing $d$\nobreakdash-representable complexes (for any fixed $d\geq 2$) is
NP-hard. 
It is entirely possible that recognizing $3$\nobreakdash-realizable convex codes is in fact an undecidable problem.

\bibliographystyle{plain}
\bibliography{convexcodes2d}
\end{document}